\documentclass[11pt]{amsart}
\usepackage{amsmath,amsthm,amssymb,url}
\usepackage[all]{xy}

\newtheorem*{theorem*}{Theorem}
\newtheorem*{proposition*}{Proposition}
\newtheorem*{corollary*}{Corollary}

\newtheorem{theorem}{Theorem}[section]

\newtheorem{proposition}[theorem]{Proposition}

\theoremstyle{definition}

\theoremstyle{remark}
\newtheorem{remark}[theorem]{Remark}

\renewcommand{\bar}{\overline}

\newcommand{\g}{\mathfrak{g}}
\newcommand{\h}{\mathfrak{h}}

\newcommand{\K}{\mathbb{K}}

\newcommand{\Id}{\operatorname{Id}}

\newcommand{\MC}{\operatorname{MC}}
\newcommand{\Def}{\operatorname{Def}}
\newcommand{\Del}{\operatorname{Del}}
\newcommand{\Hom}{\operatorname{Hom}}

\newcommand{\holim}{\operatorname{holim}}

\newcommand{\Tot}{\operatorname{Tot}}

\newcommand{\Eps}{\mathcal{E}}
\newcommand{\U}{\mathcal{U}}
\newcommand{\C}{\mathbb{C}}

\begin{document}

\title[Semicosimplicial DGLAs in deformation theory]{
Semicosimplicial DGLAs in deformation theory}

\begin{abstract}We identify \v{C}ech cocycles in nonabelian (formal) group cohomology with
Maurer-Cartan elements in a suitable $L_\infty$-algebra. Applications to
deformation theory are described.

\end{abstract}


\date{}
\author{Domenico Fiorenza}
\address{\newline Dipartimento di Matematica \lq\lq Guido
Castelnuovo\rq\rq,\hfill\newline Sapienza Universit\`a di Roma,
\hfill\newline P.le Aldo Moro 5, I-00185 Roma Italy.}
\email{fiorenza@mat.uniroma1.it}
\urladdr{www.mat.uniroma1.it/people/fiorenza/}

\author{Marco Manetti}
\address{\newline Dipartimento di Matematica \lq\lq Guido
Castelnuovo\rq\rq,\hfill\newline Sapienza Universit\`a di Roma,
\hfill\newline P.le Aldo Moro 5, I-00185 Roma Italy.}
\email{manetti@mat.uniroma1.it}
\urladdr{www.mat.uniroma1.it/people/manetti/}

\author{Elena Martinengo}
\address{\newline Dipartimento di Matematica \lq\lq Guido
Castelnuovo\rq\rq,\hfill\newline Sapienza Universit\`a di Roma,
\hfill\newline P.le Aldo Moro 5, I-00185 Roma Italy.}
\email{martinengo@mat.uniroma1.it}
\urladdr{www.mat.uniroma1.it}

\maketitle
\section*{Introduction}

Let $X$ be a smooth algebraic variety over an algebraically closed field $\mathbb{K}$ of characteristic 0 and let $\mathcal{U}=\{U_i\}$ be an affine open covering of $X$.  The infinitesimal deformations of $X$ are controlled by the alternating \v{C}ech complex of the tangent sheaf
\[ C^*(\mathcal{U}, \mathcal{T}_X):\quad 0\to \prod_i \Gamma(U_i, \mathcal{T}_X)\to \prod_{i<j} \Gamma(U_{ij}, \mathcal{T}_X)\to\cdots\]
in at least two canonical ways.
The first is familiar in algebraic geometry  and dates back to Kodaira and Spencer: since every infinitesimal deformation of $U_i$ is trivial \cite{Sernesi}, if $(A,\mathfrak{m}_A)$ is a local Artin algebra with residue field $\mathbb{K}$,
every  deformation of $X$ over $Spec(A)$ is obtained by patching the schemes $U_i\times Spec(A)$ along the open subsets $U_{ij}\times Spec(A)$ and therefore by a sequence $\theta_{ij}$ of automorphisms of the trivial deformations $U_{ij}\times Spec(A)$  satisfying the cocycle condition on triple intersections.
Since the base field has characteristic 0, for every open subset $U\subset X$, the group of automorphisms of the trivial deformation $U\times Spec(A)$ is isomorphic to $\exp(\Gamma(U, \mathcal{T}_X)\otimes\mathfrak{m}_A)$ and the deformation equation of $X$ is $e^{d_{ij}}e^{-d_{kj}}e^{d_{ki}}=1$ for $k<i<j$, or equivalently
\begin{equation}\label{equ.deformationBCH} 
d_{ij}\bullet (-d_{kj})\bullet d_{ki}=0,\qquad k<i<j,\qquad
d_{ab}\in \Gamma(U_{ab},T_X)\otimes\mathfrak{m}_A,
\end{equation}
where $\bullet$ is the Baker-Campbell-Hausdorff product and $\mathfrak{m}_A$ is the maximal ideal of $A$. The isomorphism classes of deformations of $X$ over $Spec(A)$ are therefore 
classified by the first cohomology set $H^1(\mathcal{U}, \exp(\mathcal{T}_X\otimes\mathfrak{m}_A))$.

The second way $C^*(\mathcal{U}, \mathcal{T}_X)$ is related to deformations of $X$ is the subject of this paper. According the differential graded Lie algebras approach to deformation theory in characteristic zero, deformations of $X$ are governed by a suitable differential graded Lie algebra, or more in general, by an $L_\infty$-algebra. And indeed, looking at the problem from the homotopical point of view, following \cite{bousfield-kan} and \cite{hinich} the  isomorphism classes of infinitesimal deformation of $X$ are governed by the homotopy limit of the cosimplicial Lie algebra of \v{C}ech cochains of vector fields over  the covering $\mathcal{U}$. Then, using standard homology perturbation theory we get a canonical sequence of higher brackets $\{\mu_n\}$ on the complex $C^*(\mathcal{U}, \mathcal{T}_X)$ inducing an
$L_{\infty}$ structure governing infinitesimal deformations of $X$ \cite{chenggetzler,Fio-Man}. 
The way this $L_\infty$ structure governs deformations of $X$ turns out to be the most natural possible: we show that the solutions of the deformation equation (\ref{equ.deformationBCH}) coincide, as a subset of $C^1(\mathcal{U},\mathcal{T}_X)\otimes\mathfrak{m}_A=\prod_{i<j}\Gamma(U_{ij},T_X)\otimes\mathfrak{m}_A$, with the solutions of the  Maurer-Cartan equation 
\begin{equation}\label{equ.deformationTOT} 
\sum_{n=1}^{\infty} \frac{\mu_n(x^{\otimes n})}{n!}=0,\qquad x\in C^1(\mathcal{U},\mathcal{T}_X)\otimes\mathfrak{m}_A
\end{equation}
for the $L_\infty$-algebra $(C^*(\mathcal{U}, \mathcal{T}_X),\mu_1,\mu_2,\dots)$, and that  two solutions of the deformation equation give isomorphic deformations if and only if they are homotopy equivalent Maurer-Cartan elements. In other words, we refine Hinich equivalence between the  groupoid $\Def_X$ of deformations of $X$ and the Deligne groupoid of $\holim\,\Gamma(\mathcal{U},\mathcal{T_X})$ to a natural isomorphism of groupoids between $\Def_X$ and the Poincar\'e groupoid of $C^*(\mathcal{U}, \mathcal{T}_X)$.

The above constructions extend naturally to any semicosimplicial Lie algebra 
 \[
\xymatrix{ {{\mathfrak g}_0}
\ar@<2pt>[r]\ar@<-2pt>[r] & { {\mathfrak g}_1}
      \ar@<4pt>[r] \ar[r] \ar@<-4pt>[r] & { {\mathfrak g}_2}
\ar@<6pt>[r] \ar@<2pt>[r] \ar@<-2pt>[r] \ar@<-6pt>[r]&
\cdots}
\]
over a field of characteristic 0. As an immediate geometric application, one obtains canonical $L_\infty$-algebras governing classical deformation problems such as infinitesimal deformations of a locally free sheaf $\Eps$ of $\mathcal{O}_X$-modules and infinitesimal deformations of the pair $(X, \Eps)$.
\bigskip

The paper is organized as follows: in the first three sections we fix notations, review results by Hinich \cite{hinich,hinich-algebras} and Getzler \cite{getzler}, and recall the Thom-Whitney model for the homotopy limit of a semicosimplicial DGLA; in Section 4 we prove our main result, postponing the details of the construction of the $L_\infty$-algebra to the Appendix; finally, in Section 5 geometric applications are described.
\bigskip

We thank the Referee for useful comments and criticism 
which helped us a lot in improving this paper. In particular the Referee 
should be credited for the proof of the Hinich-Getzler  homotopy equivalence 
theorem we shortly describe in Section \ref{Section.Hinich-Getzler}; namely, this proof had been sketched in 
such a detail in one of the Referee's letters that our only contribution 
to it may be considered the eventual introduction of errors and 
imprecisions. We thank Donatella Iacono for several interesting conversations on the subject of this paper.
\bigskip

Throughout the paper, $\K$ is a fixed field of characteristic zero. 

\section{Deligne groupoids and descent data}\label{Section.Deligne}
For any DGLA $\g$, the Deligne groupoid of $\mathfrak{g}$ is the formal groupoid 
\[ {\rm Del}(\g): {\bf Art}_{\K} \to {\bf Grpd}, \] 
which associates to every local artinian $\K$-algebra $A$ with maximal ideal $\mathfrak m_A$ and residue field $\K$ the action groupoid 
\[{\rm Del}(\g)(A)=\MC(\g\otimes \mathfrak m_A)//\exp(\g^0\otimes \mathfrak m_A),
\]
where  $\MC(\g\otimes \mathfrak m_A)$ is the set of Maurer-Cartan elements of the DGLA $\g\otimes \mathfrak m_A$, i.e., the set of elements $x\in   \g^1\otimes \mathfrak m_A$ which satisfy the Maurer-Cartan equation $dx +\frac{1}{2}[x,x]=0$, and the group $\exp(\g^0\otimes \mathfrak m_A)$ acts on Maurer-Cartan elements via gauge transformations $e^a\colon x\mapsto e^a*x$. In the above formulas, $\mathfrak{g}^i$ denotes the subspace of degree $i$ elements of the DGLA $\g$.
More explicitly, ${\rm Del}(\g)(A)$ is the small groupoid whose set of objects is the set $\rm{Del}(\g)(A)$ and with 
\[
\Hom_{{\rm Del}(\g)(A)}(x,x') = \{ a \in \g^0\otimes \mathfrak m_A \mid e^a*x =x'   \}. 
\]
\par
Let now  ${\mathfrak g}^\Delta$ be a \emph{semicosimplicial differential graded Lie algebra}, i.e.,  a
covariant functor $\mathbf{\Delta_{mon}}\to
\mathbf{DGLAs}$, from the category
$\mathbf{\Delta_{mon}}$, whose objects are finite
ordinal sets and whose morphisms are order-preserving injective
maps between them, to the category of DGLAs. In other words, ${\mathfrak g}^\Delta$ is a diagram
 \[
\xymatrix{ {{\mathfrak g}_0}
\ar@<2pt>[r]\ar@<-2pt>[r] & { {\mathfrak g}_1}
      \ar@<4pt>[r] \ar[r] \ar@<-4pt>[r] & { {\mathfrak g}_2}
\ar@<6pt>[r] \ar@<2pt>[r] \ar@<-2pt>[r] \ar@<-6pt>[r]&
\cdots}
\]
where each ${\mathfrak g}_i$ is a DGLA, and for each
$i>0$ there are $i+1$ morphisms of DGLAs
\[
\partial_{k,i}\colon {\mathfrak g}_{i-1}\to {\mathfrak
g}_{i},
\qquad k=0,\dots,i,
\]
such that
$\partial_{k+1,i+1}\partial_{l,i}=\partial_{l,i+1}\partial_{k,i}$,
for any
$k\geq l$.
\par
Applying the Deligne functor to ${\mathfrak g}^\Delta$, we obtain a semicosimplicial formal groupoid ${\rm Del}({\mathfrak g}^\Delta)$. The \emph{groupoid of descent data} for ${\rm Del}({\mathfrak g}^\Delta)$, i.e., the homotopy limit of the diagram ${\rm Del}({\mathfrak g}^\Delta)$, is easily desribed: it is the formal groupoid whose objects are pairs $(l,m)$ with $l$ an object in ${\rm Del}({\mathfrak g}_0)$ and $m$ a morphism in ${\rm Del}({\mathfrak g}_1)$ between $
\partial_{0,1}l$ and $\partial_{1,1}l$, such that the three images of $m$ via the maps $\partial_{i,2}$ are the edges of a 2-simplex in the nerve of ${\rm Del}({\mathfrak g}_2)$; morphisms between $(l_0,m_0)$ and $(l_1,m_1)$ are morphisms $a$ in ${\rm Del}({\mathfrak g}^0)$ between $l_0$ and $l_1$ making the diagram
\[
\xymatrix{
\partial_{0,1}l_0\ar[r]^{e^{m_0}}\ar[d]_{e^{\partial_{0,1}a}}&\partial_{1,1}l_0\ar[d]^{e^{\partial_{1,1}a}}\\
\partial_{0,1}l_1\ar[r]^{e^{m_1}}&\partial_{1,1}l_1
}
\]
 in ${\rm Del}({\mathfrak g}^1)$ commute. In a more compact and explicit (but less enlighting) form, for any local Artin algebra $A$, $\holim \Del(\g^\Delta)(A)$ is the small groupoid whose set of objects is the set
\[ \left\{ (l,m)\in
(\g_0^1 \oplus \g^0_1) \otimes \mathfrak m_A \left|
\begin{array}{l}
dl+\frac{1}{2}[l,l]=0,\\ e^{m}*\partial_{0,1}l=\partial_{1,1}l, \\
e^{\partial_{0,2}m} e^{-\partial_{1,2}m}e^{\partial_{2,2}m} =1
\end{array} \right.\right\}\]
and whose set of morphisms between two objects $(l_0,m_0)$ and $(l_1,m_1)$ is the set
\[ \left \{a \in \g^0_0\otimes \mathfrak m_A \left| \begin{array}{l}
e^a * l_0=l_1\\
e^{- m_0}e^{-\partial_{1,1}a}e^{m_1}e^{\partial_{0,1}a}=1 \end{array}\right. \right\}. \]
Note that in the case $\g^\Delta$ is a semicosimplicial Lie algebra, i.e., all the DGLAs $\mathfrak{g}_i$ are concentrated in degree zero, the set of objects of the groupoid of descent data for $\Del(\g^\Delta)$ is the set of nonabelian 1-cocycles
 
\[ Z^1(\exp(\g_1))(A)= \{ m\in \g_1 \otimes \mathfrak m_A \mid
e^{\partial_{0,2}(m)} e^{-\partial_{1,2}(m)}
e^{\partial_{2,2}(m)}=1 \},  \]
while the set of its morphisms between the two objects $m_0$ and $m_1$ is the set
of 1-bordisms between $m_0$ and $m_1$:
\[ \{ a\in
\g_0\otimes \mathfrak m_A \mid e^{-\partial_{1,1}(a)} e^{m_1} e^{\partial_{0,1}(a)}=e^{m_0} \}.\]
Therefore, the groupoid of descent data is a groupoid refinement of the nonabelian group cohomology set $H^1(\exp({\g_1 \otimes \mathfrak m_A}))$.

The name `groupoid of descent data' has an evident geometric origin. Consider, for instance, the classical problem of deformations of a smooth complex variety $X$.  
The groupoid-valued functor of infinitesimal deformations of $X$ is the functor
\[ \Def_X: {\bf Top}(X)\times {\bf Art}_\C \to {\bf Grpd} \]
which maps an open subset $U$ of $X$ and a local Artin $\C$-algebra $A$ to the groupoid whose objects are deformations of $U$ over the pointed formal scheme ${\rm Spec}(A)$, and whose morphisms are isomorphisms between deformations. It is well known that $\Def_X$ is a formal stack over the site ${\bf Top}(X)$. In particular, the formal groupoid of global deformations of $X$ 
is the groupoid of descent data for local deformations:
\[ \Def_X(X) \simeq \displaystyle \holim_{U\in \Delta_\U} \Def_{X}(U), \]
where $\Delta_\U$ is the semisimplicial object in ${\bf{Top}}_X$
associated with an open cover $\U$ of $X$. If $U$ is an affine (or Stein in the complex analytic case) open subset of $X$ and $A$ is a local Artin $\mathbb{C}$-algebra, then the groupoid $\Def_X(U;A)$ is equivalent to the action groupoid $*//\exp(\mathcal{T}_X(U)\otimes \mathfrak{m}_A)$, where $\mathcal{T}_X$ is the tangent sheaf of $X$. That is, 
\[
\Def_X(U)\simeq {\rm Del}(\mathcal{T}_X(U)).
\] 
This equivalence is compatible with restriction maps; therefore if we consider the 
\v{C}ech
semicosimplicial Lie algebra $ \mathcal{T}_X(\mathcal{U})$ associated with the tangent sheaf and an affine open cover $\U$ of $X$, i.e., 
\[
\xymatrix{ {\prod_i\mathcal{T}_X(U_i)}
\ar@<2pt>[r]\ar@<-2pt>[r] & {
\prod_{i<j}\mathcal{T}_X(U_{ij})}
      \ar@<4pt>[r] \ar[r] \ar@<-4pt>[r] &
      {\prod_{i<j<k}\mathcal{T}_X(U_{ijk})}
\ar@<6pt>[r] \ar@<2pt>[r] \ar@<-2pt>[r] \ar@<-6pt>[r]& \cdots},
\]
we obtain a commutative diagram of formal groupoids
\[
\xymatrix{ {\rm Del}({\prod_i\mathcal{T}_X(U_i))}\ar[d]^{\wr}
\ar@<2pt>[r]\ar@<-2pt>[r] & {
\prod_{i<j}{\rm Del}(\mathcal{T}_X(U_{ij}))}\ar[d]^{\wr}
      \ar@<4pt>[r] \ar[r] \ar@<-4pt>[r] &
      {\prod_{i<j<k}{\rm Del}(\mathcal{T}_X(U_{ijk}))}\ar[d]^{\wr}
\ar@<6pt>[r] \ar@<2pt>[r] \ar@<-2pt>[r] \ar@<-6pt>[r]& \cdots
\\
 {\prod_i\Def_X(U_i)}
\ar@<2pt>[r]\ar@<-2pt>[r] & {
\prod_{i<j}\Def_X(U_{ij})}
      \ar@<4pt>[r] \ar[r] \ar@<-4pt>[r] &
      {\prod_{i<j<k}\Def_X(U_{ijk})}
\ar@<6pt>[r] \ar@<2pt>[r] \ar@<-2pt>[r] \ar@<-6pt>[r]& \cdots
},
\]
where all the vertical arrows are equivalences. It follows that the induced map between homotopy limits is an equivalence, and so $\holim_{U\in \Delta_\U} \Del(\mathcal{T}_X(U))$ is equivalent to the groupoid of descent data for infinitesimal deformations of $X$ (and so it is equivalent to the groupoid of infinitesimal deformations of $X$). In particular, one recovers the classical description of the set of isomorphism classes of infinitesimal deformations of $X$ over ${\rm Spec}(A)$ as 
\[
H^1(X;\exp(\mathcal{T}_X\otimes\mathfrak{m}_A)).
\]
Closely related examples are infinitesimal deformations of a locally free sheaf $\Eps$ of $\mathcal{O}_X$-modules and infinitesimal deformations of the pair $(X, \Eps)$. The above argument verbatim applies to these cases: the sheaves of Lie algebras involved are the sheaf $\Eps{nd}(\Eps)$ of endomorphisms of $\Eps$ and the sheaf ${\mathcal D}^1(\Eps)$ of first order
differential operators on $\Eps$ with scalar principal symbol, respectively.

\section{Overview of Hinich and Getzler results}\label{Section.Hinich-Getzler}

One obtains a much more versatile theory moving from formal groupoids to formal $\infty$-groupoids, i.e., to functors ${\bf Art}_{\K} \to {\bf Kan\ complexes}$.
There is an obvious functorial way to produce an $\infty$-groupoid out of a groupoid, which is taking its nerve, i.e., the simplicial set whose $k$-simplices are the $k$-tuples of composable morphisms in the groupoid. This way, the Deligne groupoid of a DGLA $\mathfrak{g}$ is promoted to the formal $\infty$-groupoid
\[
N({\rm Del}(\g)): {\bf Art}_{\K} \to {\bf Kan\ complexes}. 
\]  
It is a remarkable insight of Sullivan \cite{sullivan}, then fully investigated by Hinich and Getzler \cite{hinich,getzler} that, when $\mathfrak{g}$ is concentrated in nonnegative degrees,  $N({\rm Del}(\g))$ admits a simple description in terms of differential forms on algebraic simplices.
More precisely, for every $n\ge 0$, denote by $\Omega_n$ the differential graded
commutative algebra of polynomial differential forms on the
standard $n$-simplex $\Delta^n$:
\[ \Omega_n=\frac{\mathbb{K}[t_0,\ldots,t_n,dt_0,\ldots,dt_n]}{(\sum t_i-1,\sum dt_i)},\]
and, or any DGLA $\g$, consider the functor
\[ \MC_{\infty}(\g): {\bf Art}_{\K} \to {\bf SSet},\]
which associates to every local artinian $\K$-algebra $(A,\mathfrak m_A)$ the simplicial set $\MC((\g\otimes\otimes  \Omega_{\bullet})\otimes \mathfrak m_A)$. Then Hinich proves
\begin{proposition}[\cite{hinich}, Proposition 2.2.3] \label{Proposition.MC e Del}
Let $\g$ be a DGLA concentrated in nonegative degree. Then there is a natural homotopy equivalence:
\[ \Phi: \MC_\infty(\g) \stackrel{\simeq} {\longrightarrow}N (\Del(\g)). \]
In particular, $\Phi$ induces an equivalence of formal groupoids $\pi_{\leq 1} \MC_\infty(\g)\simeq 
\Del(\g)$.
\end{proposition}
The map $\Phi$ has a simple explicit expression on the lower skeleta. Namley both sides have the same 0-skeleton, and $\Phi$ is the identity on this set; on $1$-simplices, $\Phi$ it is given as follows: every element in $\MC(\g \otimes \Omega_1)(A)$ can be uniquely written as $e^{p(t)}*\xi$, where $p(t) \in  (\g^0 \otimes \Omega_1^0)\otimes\mathfrak{m}_A$ satisfies $p(0)=0$ and $\xi \in \MC(\g)(A)$, and $\Phi$ maps $e^{p(t)}*x$ to $p(1)$. 

\bigskip

The construction of $\MC_\infty$ naturally generalizes to $L_\infty$-algebras: one just has to replace the Maurer-Cartan functor of a DGLA with its version for $L_\infty$-algebras; nalmely, with the functor
\[ \MC(\g): {\bf Art}_{\K} \to {\bf Set},\]
mapping $(A,\mathfrak m_A)$, to the set
\[  
\MC(\g)(A)=\left\{ x\in \g^1\otimes \mathfrak m_A \;\strut\left\vert\;
\sum_{n\ge1}\frac{[x,x,\dots,x]_n}{n!}=0\right.\right\},
\]
where 
\[
[\,,\dots,\,]_n\colon \wedge^n\mathfrak{g}\to\mathfrak{g}[2-n]
\]
are the brackets of the $L_\infty$-algebra $\mathfrak{g}$. 

Hinich and Getzler prove that the simplicial set valued functor $\MC_\infty(\g)$ takes values in Kan complexes, for any  DGLA  and, more generally, for any $L_\infty$-algebra $\mathfrak{g}$. Moreover they prove that the functor $\MC_\infty$ sends surjective morphisms of DGLAs  to fibrations of Kan complexes and quasi-isomorphisms of DGLAs to equivalences of Kan complexes; both these statements generalize to $L_\infty$-algebras. Since surjective morphisms and quasi-isomorphisms are the fibrations and the weak-equivalences in the standard model structure on the category of DGLAs (see \cite{pridham}), Hinich and Getzler results are a first major step toward the 
formalization of the following folk statement: \emph{the $(\infty,1)$-category of $L_\infty$-algebras is equivalent to the $(\infty,1)$-category of formal $\infty$-groupoids}, which naturally generalizes the well known equivalence between the category of Lie algebras and the category of formal Lie groups (in characteristic zero). Recently, a further major step in this direction has been made by Pridham, who proves in \cite{pridham} that the homotopy categories of DGLAs and of  $L_\infty$-algebras is equivalent to a certain category of $\bf SSet$-valued functors of Artin rings (`geometric' deformation functors, see also \cite{manetti}).

However, in the present paper we are not interested in the above correspondence, but rather in another of Hinich results from \cite{hinich}; namely that the functor
\[
\MC_{\infty}\colon \text{\bf DGLAs}^{\geq 0}\to \text{\bf Formal $\infty$-Grpds},
\]
where $\text{\bf DGLAs}^{\geq 0}$ denotes the full subcategory of {\bf DGLAs} consisting of differential graded Lie algebras concentrated in nonnegative degrees, commutes with homotopy limits. In other words, Hinich proves the following
\begin{proposition}[\cite{hinich}, Theorem 4.1] \label{Th.MC commuta con holim}
Let $\g^{\Delta}$ be a semicosimplicial DGLA such that $\mathfrak{g}_k$ is concentrated in nonnegative degrees for any $k\geq 0$. Then there is a natural equivalence of functors
\[  \MC_{\infty}(\holim \g^{\Delta}) \simeq \holim \MC_{\infty}(\g^{\Delta}).\]
\end{proposition}
Let us sketch a proof  
of this result  by induction on the length of the descending cental series 
of the nilpotent semicosimplicial Lie algebra $\mathfrak{g}^\Delta\otimes\mathfrak{m}_A$, verbatim following the argument in \cite{getzler}.  To avoid cumbersome notations, let us write $\h^{\Delta}$ for $\mathfrak{g}^\Delta\otimes\mathfrak{m}_A$ and denote by $\{F^i\h^\Delta\}_{i\geq 1}$ the descending central series of $\h^\Delta$, i.e., the filtration defined by
 $F^1\h^{\Delta}=\h^\Delta$ and
\[ F^i\h^\Delta= \sum_{j+k=i} [F^{j}\h^\Delta, F^k\h^\Delta], \qquad \text{for $i>1$} \] 
(note that nilpotence degree of the $\h_k$'s is uniformly bounded by the nilpotence degree of $\mathfrak m_A$). Also, let us write 
$\MC_\infty(\mathfrak{h}^\Delta)$ for $\MC_\infty(\mathfrak{g}^\Delta)(A)$.
\par
The base of the induction is $F^2\h^\Delta=0$, i.e. an abelian $\h^{\Delta}$. In this case, the homotopy equivalence $\MC_{\infty}(\holim \h^{\Delta}) \simeq \holim \MC_{\infty}(\h^{\Delta})$ is just the Eilenberg-Zilber-Dold-Kan-Puppe theorem \cite{goerss-jardine}, as one immediately realizes by noticing that an abelian DGLA concentrated in nonnegative degrees is the same thing as a complex of abelian groups concentrated in nonnegative degrees; that the total complex of the bicomplex associated with a semicosimplicial complex $\mathfrak{h}^\Delta$ of abelian groups is a model for $\holim(\mathfrak{h}^\Delta)$ in the model category of complexes; that the diagonal of a bisimplicial set $X_{\bullet,\bullet}$ is naturally homotopy equivalent to the simplicial set $\holim_{n} X_{n,\bullet}$ (see, e.g., Chapter 4 in \cite{goerss-jardine}); and, finally, that looking at a complex of abelian groups $\mathfrak{h}$ as an abelian DGLA, the Dold-Kan correspondence $\mathfrak{h}\mapsto \Hom_{{\bf Ch_+(Ab)}}({\rm Chains}(\Delta_\bullet),\h)$ is nothing but the functor $\mathfrak{h}\mapsto \MC(\mathfrak{h}\otimes\Omega_\bullet)$. Note that the original Dold-Kan correspondence is formulated in terms of cosimplicial abelian groups rather than in terms of semicosimplicial abelian groups as we did here: we implicitly used the fact that every semisimplicial object induces a simplicial object by Kan extension.
\par
Next, if the length of the descending central series of  $\mathfrak{h}^\Delta$ is equal to $i> 1$, consider the following diagram of simplicial sets and morphisms, induced by the universal property of homotopical limits:
\[ \xymatrix{ \MC_\infty(\holim (\ F^{2}\h^\Delta)) \ar[d]^\wr\ar[r]& \MC_\infty(\holim\ (\h^\Delta)) \ar[d]\ar[r] &\MC_\infty(\holim \ (\h^\Delta/F^{2}h^\Delta)) \ar[d]^\wr\\
\holim\ \MC_\infty(F^{2}\h^\Delta) \ar[r]& \holim\ \MC_\infty(\h^\Delta) \ar[r] &\holim\ \MC_\infty(\h^\Delta/F^{2}h^\Delta). }\]
Since the length of descending central series of $F^2\mathfrak{h}^\Delta$ is at most $i-1$ and $\mathfrak{h}^\Delta/F^2\mathfrak{h}^\Delta$ is abelian, the inductive hypothesis applies and so the leftmost and rightmost vertical morphisms are homotopy equivalences. Moreover the horizontal sequences are fibration sequences (homotopy limits preserve fibrations), so by the long  exact sequence of homotopy groups also the central vertical morphism is a homotopy equivalence.  
\begin{remark}
Following Getzler, one could extend the above result to semicosimplicial $L_\infty$-algebras concentrated in nonnegative degree, and so, in particular, to DGLAs whose cohomology is concentrated in nonnegative degree.
\end{remark}

Combining Propositions \ref{Proposition.MC e Del} and \ref{Th.MC commuta con holim}, Hinich obtains that, for semicosimplicial DGLAs concentrated in nonnegative degrees, the Deligne functor commutes with homotopy limits, i.e., there is a natural equivalence of formal groupoids
\[ \Del(\holim \g^\Delta) \simeq \holim \Del(\g^\Delta). \]
The main result of the present paper consists in showing that, when $\mathfrak{g}^\Delta$ is a semicosimplicial Lie algebra, by choosing a suitable model $\widetilde{\Tot}(\mathfrak{g}^\Delta)$ for $\holim \g^\Delta$ in the category of $L_\infty$-algebras one can refine the above equivalence to an \emph{isomorphism} of formal groupoids
\[
\pi_{\leq 1}\MC_\infty(\widetilde{\Tot}(\mathfrak{g}^\Delta))\cong \holim \Del(\g^\Delta).
\] 
In other words, descent data for $\Del(\g^\Delta)(A)$ are identified with Maurer-Cartan elements of the $L_\infty$-algebra $\widetilde{\Tot}(\mathfrak{g}^\Delta\otimes\mathfrak{m}_A)$ and isomorphisms between descent data are identified with (homotopy classes of) homotopy equivalences between Maurer-Cartan elements.

\section{The Thom-Whitney model and homotopy transfer of $L_\infty$-structures}\label{Section.TW}
Our first step towards the construction of the $L_\infty$-algebra $\widetilde{\Tot}(\mathfrak{g}^\Delta\otimes\mathfrak{m}_A)$ consists in considering the Thom-Whitney model for the homotopy limit of $\g^{\Delta}$ in the model category of DGLAs, i.e., the end
\[
\operatorname{Tot}_{TW}(\mathfrak{g}^\Delta)=\int_{n\in\mathbb{N}}\Omega_n\otimes\mathfrak{g}_n.
\]
More explicitly, denote by $\Omega_n^i$ the degree $i$ component of $\Omega_n$ and
by $\delta^{k,n}\colon \Omega_n\to \Omega_{n-1}$, $k=0,\ldots,n$,
the face maps. Then we have  natural morphisms of bigraded DGLAs
\[ \delta^{k,n}\colon \Omega_n\otimes \mathfrak{g}_n\to
\Omega_{n-1}\otimes\mathfrak{g}_n,\qquad
\partial_{k,n}\colon \Omega_{n-1}\otimes \mathfrak{g}_{n-1}\to
\Omega_{n-1}\otimes\mathfrak{g}_{n}\]
for every $0\le k\le n$, and the Thom-Whitney DGLA $\operatorname{Tot}_{TW}(\mathfrak{g}^\Delta)$ is defined as the total complex of the Thom-Whitney bicomplex
\[
C^{i,j}_{TW}(\mathfrak{g}^\Delta)
=\{ (x_n)_{n\in {\mathbb
N}}\in \prod_n \Omega_n^i\otimes {\mathfrak g}_n^j
\mid \delta^{k,n}x_n=
\partial_{k,n}x_{n-1}\quad \forall\; 0\le k\le n\}
\]
(the Lie brackets on the $\g_i$'s induce a Lie bracket on the total complex, making it a DGLA).
\par
Since the model structure on the category $\bf{DGLA}$ is defined as for differential complexes (with  surjective morphisms as fibrations), the complex underlying the DGLA  $\Tot_{TW}(\g^\Delta)$ is a model for $\holim \g^\Delta$ in the model category of differential complexes (where, with a little abuse of notation, we have denoted by the same symbol the semicosimplicial complex underlying $\g^\Delta$). However, in the model category of complexs there is a smaller and more tractable model for the homotopy limit of $\g^\Delta$; namely, its total complex. Let us briefly recall its construction. If $\g^\Delta$ is a semicosimplicial complex,
with morphisms $\partial_{k,i}\colon {\mathfrak g}_{i-1}\to {\mathfrak
g}_{i}$, for $k=0,\dots,i,$ then
the maps
\[
\partial_i=\partial_{0,i}-\partial_{1,i}+\cdots+(-1)^{i}
\partial_{i,i}
\]
endow the vector space $\bigoplus_i{\mathfrak g}_i$ with the
structure of a differential complex. Moreover, each ${\mathfrak g}_i$ is a differential
complex:
$
{\mathfrak g}_i=\bigoplus_j {\mathfrak g}_i^j$, $
d_i\colon {\mathfrak g}_i^j\to{\mathfrak g}_i^{j+1}
$,
and, since the maps $\partial_{k,i}$ are morphisms of complexes,
the space
${\mathfrak g}^\bullet_\bullet=\bigoplus_{i,j}{\mathfrak
g}_i^j$
has a natural bicomplex structure, whose associated total complex $({\rm Tot}({\mathfrak g}^\Delta),d_{\Tot})$ is
\[{\rm Tot}({\mathfrak g}^\Delta)=\bigoplus_{i}{\mathfrak
g}_i[-i],\quad d_{\Tot}=\sum_{i,j}\partial_i+(-1)^jd_j.\] 
When $\g^\Delta$ is a semicosimplicial DGLA, there is unfortunately no hope that ${\rm Tot}({\mathfrak g}^\Delta)$ is a model for $\holim(\g^\Delta)$ in the model category of DGLAs: the complex $({\rm Tot}({\mathfrak g}^\Delta),d_{\Tot})$ carries
no natural DGLA structure. Yet, one can overcome this problem by going in the larger category of $L_\infty$-algebras. Indeed, $\Tot_{TW}(\g^\Delta)$ and ${\rm Tot}({\mathfrak
g}^\Delta)$ are homotopy equivalent as differential complexes, being both models for the $\holim \g^\Delta$ in the model category of complexes, and so, by homotopy transfer of structure, one can induce on $\Tot(\g^\Delta)$ an $L_\infty$-algebra structure making it quasi-isomorphic (as an $L_\infty$-algebra) to the DGLA ${\rm Tot}_{TW}({\mathfrak g}^\Delta)$.  What is relevant is that, using results by Whitney and Dupont, this $L_\infty$-algebra enrichment of the complex ${\rm Tot}({\mathfrak g}^\Delta)$ can be done in a functorial way, that is we have a commutative diagram of functors
\[
\xymatrix{& \mathbf{L_\infty\text{-algebras}}\ar[dd]\\
\mathbf{[\Delta_{mon},DGLAs]}\ar[ru]^{\widetilde{\Tot}}\ar[rd]_{\Tot}&\\
&\mathbf{Complexes}
}
\]
where the vertical arrow is the forgetful functor. More precisely, to any morphism $\mathfrak{g}^\Delta\to\mathfrak{h}^\Delta$ of semicosimplicial DGLAs, it corresponds a \emph{linear} $L_\infty$-morphism $\widetilde{\rm Tot}({\mathfrak g}^\Delta)\to \widetilde{\rm Tot}({\mathfrak h}^\Delta)$; see the Appendix for the detailed construction of $\widetilde{\rm Tot}({\mathfrak g}^\Delta)$.

\section{An isomorphism of groupoids}
Since the $n$-ary brackets of an $L_\infty$ algebra $\mathfrak{l}$ have degree $2-n$, the subspace $\mathfrak{l}^{>0}$ of elements of strictly positive degree is a sub-$L_\infty$ algebra of $\mathfrak{l}$. On the semicosimplicial DGLAs side, the corresponding operation consists in repalcing the DGLA $\g_0$ in $\g^\Delta$ with the zero DGLA, thus giving the semicosimplicial DGLA
\[ \mathfrak{g}^{\Delta_{>0}}: \ \  \xymatrix{ 0
\ar@<2pt>[r]\ar@<-2pt>[r] & \mathfrak{g}_1
      \ar@<4pt>[r] \ar[r] \ar@<-4pt>[r] & \mathfrak{g}_2
\ar@<6pt>[r] \ar@<2pt>[r] \ar@<-2pt>[r] \ar@<-6pt>[r] &
\mathfrak{g}_3 \ar@<8pt>[r] \ar@<4pt>[r] \ar[r]
\ar@<-4pt>[r] \ar@<-8pt>[r] &
\ldots }
\]
There is an obvious semicosimplicial DGLA morphism $\g^{\Delta_{>0}}\to\g^\Delta$ 
which is the identity on $\g_i$ for any $i>0$. By functoriality of the $\widetilde{\Tot}$,
 this morphism induces a linear $L_\infty$-morphism $\widetilde{\rm Tot}({\mathfrak g}^{\Delta_{>0}})\to \widetilde{\rm Tot}({\mathfrak g}^\Delta)$, which is nothing but the inclusion
 \[
 \bigoplus_{i>0}{\mathfrak
g}_i[-i]\hookrightarrow \bigoplus_{i}{\mathfrak
g}_i[-i],
 \]
so that, if $\g^\Delta$ is a semicosimplicial Lie algebra (i.e., all the $\g_k$ are concentrated in degree $0$) we have an identification
\[
\widetilde{\rm Tot}({\mathfrak g}^{\Delta_{>0}})=(\widetilde{\rm Tot}({\mathfrak g}^\Delta))^{>0}.
\]
Having introduced these notations, we can prove the main result of this paper.
\begin{theorem} \label{Th.Iso gruppoidi}
Let $\g^\Delta$ be a semicosimplicial Lie algebra. The there is a natural isomorphism of formal groupoids  
\[ \pi_{\leq 1}\MC_\infty(\widetilde{\Tot}(\g^{{\Delta}})) \cong  \holim \rm{Del} (\g^{\Delta}).
\]
\end{theorem}
\begin{proof} Let us choose $\Tot_{TW}$ as a model for the homotopy limit of a semicosimplicial DGLA. Then, by results from Section  \ref{Section.Hinich-Getzler} and by the naturality of the
 $L_\infty$-quasiisomorphism  $E_\infty\colon\widetilde{\Tot}\to \Tot_{TW}$ provided by the homotopy transfer, we have a commutative diagram of formal groupoids
\[
\xymatrix{
\pi_{\leq 1}\MC_\infty(\widetilde{\Tot}(\g^{{\Delta_{>0}}}))\ar[r]^{\sim\phantom{mi}}\ar[d]&\pi_{\leq 1}\MC_\infty({\Tot}_{TW}(\g^{{\Delta_{>0}}}))\ar[r]^{\phantom{mm}\sim}\ar[d]& \holim \rm{Del} (\g^{\Delta_{>0}})\ar[d]\\
\pi_{\leq 1}\MC_\infty(\widetilde{\Tot}(\g^{{\Delta}}))\ar[r]^{\sim\phantom{mi}}&\pi_{\leq 1}\MC_\infty({\Tot}_{TW}(\g^{{\Delta}}))\ar[r]^{\phantom{mm}\sim}& \holim \rm{Del} (\g^{\Delta})\\
}
\]
where the horizontal arrows are equivalences. The commutativity of this diagram is strict, i.e., we have a commutative diagram between the sets of objects:
\[
\xymatrix{
\MC(\widetilde{\Tot}(\g^{{\Delta_{>0}}}))\ar[r]^{E_\infty\phantom{mi}}\ar@{=}[d]&\MC({\Tot}_{TW}(\g^{{\Delta_{>0}}}))\ar[r]^{\phantom{mm}\Phi}\ar[d]& Z^1(\exp(\g_1))\ar@{=}[d]\\
\MC(\widetilde{\Tot}(\g^{{\Delta}}))\ar[r]^{E_\infty\phantom{mi}}&\MC({\Tot}_{TW}(\g^{{\Delta}}))\ar[r]^{\phantom{mm}\Phi}& Z^1(\exp(\g_1))\\
}
\]
where we used the description of the objects in $\holim{\rm Del}(\g^\Delta)$ given in Section \ref{Section.Deligne}, and $\Phi$ is the map described in Proposition \ref{Proposition.MC e Del}. More precisely, since ${\Tot}_{TW}(\g^{{\Delta}})$ is a subDGLA of $\prod_{n\geq 0}\Omega_n\otimes\g_n$, an element $x\in \MC({\Tot}_{TW}(\g^{{\Delta}}))(A)$ will have a component $x_1$ in $\MC(\Omega_1\otimes\g_1)(A)$. The element $x_1$ can be uniquely written as $e^{p(t)}*\xi_1$, with $p(t) \in   (\Omega_1^0\otimes \g^0)\otimes\mathfrak{m}_A$ satisfying $p(0)=0$ and $\xi_1\in \MC(\g_1)(A)$; one has $\Phi(x)=p(1)$. 
\par
Therefore, to prove that  $\pi_{\leq 1}\MC_\infty(\widetilde{\Tot}(\g^{{\Delta}}))\xrightarrow{\sim} \holim \rm{Del} (\g^{\Delta})$ is an isomorphism,  we are reduced to show that the equivalence 
\[
\pi_{\leq 1}\MC_\infty(\widetilde{\Tot}(\g^{{\Delta_{>0}}}))\xrightarrow{\sim} \holim \rm{Del} (\g^{\Delta_{>0}})
\]
is an isomorphism. And this is obvious, since both the groupoid on the left and on the right hand side are discrete (i.e., their only morphisms are identities), and so any equivalence between them is actually an isomorphisms.
\par
To see that $\holim \rm{Del} (\g^{\Delta_{>0}})$ is discrete, just look at the explicit description given in Section \ref{Section.Deligne}. Finally, recall that $\widetilde{\Tot}(\g^{{\Delta_{>0}}})=(\widetilde{\Tot}(\g^{\Delta}))^{>0}$. Therefore a $1$-simplex $\sigma$ in $\MC_\infty(\widetilde{\Tot}(\g^{{\Delta_{>0}}}))(A)$ has the form $\sigma=x(t)$, with $x(t)\in\g_1[t]\otimes\mathfrak{m}_A$ satisfying the Maurer-Cartan equation for the $L_\infty$-algebra $(\widetilde{\Tot}(\g^{\Delta})\otimes\Omega_1)\otimes\mathfrak{m}_A$. Written out explicitly, this equation has the form
\[
\frac{d\, x(t)}{dt}+d_{\Tot}x(t)+\sum_{n\geq 2}\frac{[x(t),x(t),\dots,x(t)]_n}{n!}=0
\]
so, looking at the degree 1 component with respect to the $\Omega_1$-grading, we found 
\[
\frac{d\, x(t)}{dt}=0,
\]
i.e. $\sigma$ is a constant path.
\end{proof}
\begin{remark}
For $\g^\Delta$ a semicosimplicial Lie algebra, let 
\[
\g_1\colon \mathbf{Art}_\K\to \mathbf{Sets}
\]
be the functor of Artin rings defined by $(A,\mathfrak{m}_A)\mapsto \g_1\otimes\mathfrak{m}_A$. Then both $\MC(\widetilde{\Tot}(\g^{{\Delta}}))$ and $Z^1(\exp(\g_1))$ are functors over $\g_1$, and one can show that the natural transformation 
\[
\Phi\circ E_\infty\colon \MC(\widetilde{\Tot}(\g^{{\Delta}}))\to Z^1(\exp(\g_1))
\]
is a natural transformation of functors over $\g_1$ (see the Appendix). In more colloquial terms, for any local Artin algebra $(A,\mathfrak{m}_A)$, the $1$-cocycles with coefficients in $\exp(\g_1\otimes\mathfrak{m}_A)$ and the Maurer-Cartan elements of the $L_\infty$-algebra $\widetilde{\Tot}(\g^{{\Delta}})\otimes\mathfrak{m}_A$ are \emph{the same} subset of $\g_1\otimes\mathfrak{m}_A$.
\end{remark}

\section{Geometric applications}
Let $\mathcal{L}$ be a sheaf of Lie algebras on a topological space $X$. Then, for any open cover $\mathcal{U}$ of $X$ we have the \v{C}ech semicosimplicial Lie algebra $\mathcal{L}({\mathcal U})$
\[
\xymatrix{ {\prod_i\mathcal{L}(U_i)}
\ar@<2pt>[r]\ar@<-2pt>[r] & {
\prod_{i<j}\mathcal{L}(U_{ij})}
      \ar@<4pt>[r] \ar[r] \ar@<-4pt>[r] &
      {\prod_{i<j<k}\mathcal{L}(U_{ijk})}
\ar@<6pt>[r] \ar@<2pt>[r] \ar@<-2pt>[r] \ar@<-6pt>[r]& \cdots}.
\]
If the cover $\mathcal{U}$ is acyclic for $\mathcal{L}$, then $\holim\mathcal{L}(\mathcal{U})\simeq \mathbf{R}\Gamma(X;\mathcal{L})$, so a suitable model for the module of derived global sections of $\mathcal{L}$ has natural structure of DGLA. For instance, we can choose $\Tot_{TW}(\mathcal{L}(\mathcal{U}))$ as such a model. Therefore, summing up results from the previous sections, we obtain that the groupoid of descent data for $\Del(\mathcal{L}(\mathcal{U}))$ is equivalent to the Deligne groupoid
\[
\Del(\mathbf{R}\Gamma(X;\mathcal{L})),
\]
see \cite{hinich-algebras}, 
and isomorphic to the formal groupoid
\[
\pi_{\leq 1}\MC_\infty(\widetilde{\Tot}(\mathcal{L}(\mathcal{U}))).
\]
Going from groupoids to their sets of isomorphisms classes of objects, we obtain an isomorphism of functors $\mathbf{Art}_\K\to \mathbf{Sets}$
\[
\pi_0\Del(\mathcal{L}(\mathcal{U}))\xrightarrow{\sim} \pi_{0}\MC_\infty(\widetilde{\Tot}(\mathcal{L}(\mathcal{U}))).
\]
The functor on the left hand side is $
H^1(X;\exp(\mathcal{L}))$,
whereas the
functor on the right hand side is the classical \emph{deformation functor} associated to the $L_\infty$-algebra $\widetilde{\Tot}(\mathcal{L}(\mathcal{U}))$, see, e.g., \cite{manetti}. Therefore, we obtain at once the following two results:
\begin{itemize}
\bigskip
\item the functor $H^1(X;\exp(\mathcal{L}))$ is a deformation functor, with tangent space $H^1(X;\mathcal{L})$ and obstruction space contained in $H^2(X;\mathcal{L})$;
\bigskip
\item if $\varphi\colon L_1\to L_2$ is a morphism of sheaves of Lie algebras over $X$, then $H^2(\varphi)$ maps the obstruction space of $H^1(X;\exp(\mathcal{L}_1))$ to the obstruction space of $H^1(X;\exp(\mathcal{L}_2))$.
\end{itemize}
\bigskip
As an illustrative example, let $\mathcal{E}$ be a locally free sheaf of $\mathcal{O}_X$-modules on a smooth variety $X$, and let $\mathcal{E}nd(\mathcal{E})$ be its sheaf of endomorphisms. Then $H^1(X;\exp(\mathcal{E}nd(\mathcal{E})))$ is the classical functor of deformations of $\mathcal{E}$, i.e., the deformation functor mapping a local Artin algebra $(A,\mathfrak{m}_A)$ to the set of isomorphism classes
of deformations of $\mathcal{E}$ over $\rm{Spec}(A)$. The trace 
\[
{\rm tr}\colon \mathcal{E}nd(\mathcal{E})\to \mathcal{O}_X
\]
is a morphism of sheaves of Lie algebras. Since $\mathcal{O}_X$ is abelian, so is $\widetilde{\Tot}(\mathcal{O}_X(\mathcal{U}))$, therefore $H^1(X;\mathcal{O}_X^*)$ is unobstructed. And so, the obstructions to deforming $\mathcal{E}$ are contained in 
\[
\ker\{{\rm tr}\colon H^2(X;\mathcal{E}nd(\mathcal{E}))\to H^2(X;\mathcal{O}_X)\}.
\] 

\begin{remark}
If $X$ is a paracompact Hausdorff space, one can compute the derived global sections of $\mathcal{L}$ by considering a fine resolution of $\mathcal{L}$, i.e., a fine sheaf $\mathcal{F}$ of DGLAs concentrated in nonnegative degree over $X$, together with an quasiisomorphism $\mathcal{L}\to\mathcal{F}$. Indeed, in this situation one has natural quasi-isomorphisms of DGLAs
\[
\mathbf{R}\Gamma(X;\mathcal{L})\xrightarrow{\sim} \mathbf{R}\Gamma(X;\mathcal{F}) \xleftarrow{\sim}\Gamma(X;\mathcal{F}).
\]
This way one recovers the infinitesimal version of Newlander-Nirenberg theorem, i.e., that infinitesimal deformations of a smooth complex manifold are controlled by the Kodaira-Spencer DGLA $A^{0,\bullet}(X;\mathcal{T}_X)$ of $(0,q)$-forms with coefficients in the holomorphic tangent sheaf of $X$ (see, e.g., \cite{Goldman-Millson,iacono}). 
Similarly, $A^{0,\bullet}(X;\mathcal{E}nd(\mathcal{E}))$ controls the infinitesimal deformations of 
the locally free sheaves of $\mathcal{O}_X$-modules $\mathcal{E}$, and  $A^{0,\bullet}_X({\mathcal D}^1(\Eps))$ 
controls the infinitesimal deformations of the pair $(X,\mathcal{E})$.
\end{remark}

\section{Appendix: the $L_\infty$ quasi-isomorphism $\widetilde{\Tot}(\g^\Delta)\to \Tot_{TW}(\g^\Delta)$}
In this section we follow \cite{chenggetzler} to endow the total complex of a semicosimplicial DGLA $\g^\Delta$ with a canonical $L_\infty$-algebra structure; a particular case of this construction appears in \cite{Fio-Man}. Notations from Section \ref{Section.TW} will be used throughout this Appendix.
It is a
remarkable fact, first noted by Whitney \cite{whitney}, that the
integration maps
\[ \int_{\Delta^n}\otimes \operatorname{Id}\colon \Omega_n\otimes
\mathfrak{g}_n\to
\mathfrak{g}_n[-n]\]
give a quasi-isomorphism of differential complexes
\[
I\colon (\operatorname{Tot}_{TW}({\mathfrak
g}^\Delta), d_{TW})\to (\operatorname{Tot}({\mathfrak
g}^\Delta),d_{\Tot}).
\]
Moreover, Dupont describes \cite{dupont1,dupont2} (see also
\cite{getzler,navarro}) an explicit morphism of differential complexes
\[
E\colon \operatorname{Tot}({\mathfrak
g}^\Delta)\to \operatorname{Tot}_{TW}({\mathfrak
g}^\Delta)
\]
and an explicit homotopy \[
h\colon \operatorname{Tot}_{TW}({\mathfrak
g}^\Delta)\to \operatorname{Tot}_{TW}({\mathfrak
g}^\Delta)[-1]
\]
such that
\[
IE={\rm Id}_{{\rm Tot}({\mathfrak
g}^\Delta)};\qquad
EI-{\rm Id}_{{\rm Tot}_{TW}({\mathfrak
g}^\Delta)}=[h,d_{TW}].
\]
The morphism of complexes $E:\operatorname{Tot}({\mathfrak
g}^{\Delta})\rightarrow \operatorname{Tot}_{TW}({\mathfrak
g}^{\Delta})$ is defined as follows. If $\gamma \in
\g^j_i$, the element $E(\gamma)=(E(\gamma)_n)\in
C^{i,j}_{TW}(\mathfrak{g}^\Delta)$ is given by:
\[
\left\{ \begin{array}{ll} E(\gamma)_n=0 & \mbox{if \ \ \ } n <i, \\ \\
E(\gamma)_n= i! \displaystyle \sum_{I\in I(i,n)} \omega_I \otimes \partial^{\bar I}\gamma &
\mbox{if\ \ \ } n\geq i, \end{array} \right.
\]
where $I(i,n)$ is the set of all multiindices $I=(a_0,a_1,\ldots,
a_i)\in \mathbb Z^{i+1}$, such that $0\leq a_0 < a_1< \ldots a_i
\leq n$ and, if $I\in I(i,n)$, $\bar I$ is the complementary
multiindex. If $I=(a_0,a_1,\ldots, a_i) \in I(i,n)$, we indicate
with $\omega_I$ the differential form:
\[\omega_I=\sum_{s=0}^i (-1)^s  t_{a_s} dt_{a_0}\wedge
dt_{a_1} \wedge \ldots \wedge \widehat{dt_{a_s}} \wedge \ldots
\wedge dt_{a_i}\in \Omega_n^i.\]
If $\bar I=(b_1,b_2,\ldots,
b_{n-i})$ is the complementary multiindex of $I$, and $\gamma\in
{\mathfrak g}_i^j$, we indicate with $\partial^{\bar I}\gamma$ the
element
\[\partial^{\bar I}\gamma= \partial_{b_{n-i},n} \circ \cdots \circ
\partial_{b_2,i+2}\circ \partial_{b_1,i+1}^{}\gamma\in {\mathfrak
g}^j_n.\] The homotopy $h: \operatorname{Tot}_{TW}({\mathfrak
g}^\Delta) \rightarrow \operatorname{Tot}_{TW}({\mathfrak
g}^\Delta)[-1]$ is then defined as follows: if $x=
(\eta_n\otimes \gamma_n) \in C^{i,j}_{TW}({\mathfrak g}^\Delta)$,
the image $h(x) \in C^{i-1,j}_{TW}({\mathfrak g}^\Delta)$
is given by:
\[ h(x)_n= \sum_{0\leq r < i} \sum_{I\in I(r,n)} r! \
\omega_I \wedge h_I(\eta_n) \otimes \gamma_n,\]
where $I(r,n)$
and $\omega_I$ are as before, and, if $I=(a_0,\ldots, a_r)\in
I(r,n)$, then the map $h_I$ is given by the composition
$h_I=h_{a_r}\circ \cdots \circ h_{a_0}$, where the maps
$h_a=\pi\circ \psi_a^*\colon \Omega_n^*\to \Omega_n^{*-1}$ are the
compositions of the integration over the first factor
\begin{align*}
\pi: \Omega^*([0,1]\times\Delta_n) &\to  \Omega^* (\Delta_n)=\Omega_n^* \\
 \eta(u,t_a,du,dt_a) & \mapsto \int_{u\in[0,1]} \eta(u,t_a,du,dt_a)
\end{align*}
and the pull-back by the dilation maps:
\begin{align*}
\psi_a:[0,1]\times \Delta_n &\to \Delta_n \\
 (u,t_0,\ldots t_n) &\mapsto ((1-u)t_0, \ldots, (1-u)t_a+u, \ldots (1-u)t_n).
\end{align*}
More explicitly, the map $h_a: \Omega_n^*\rightarrow \Omega_n^{*-1}$
 is given by:
\[h_a(\eta(t_0,\ldots,t_n,dt_0,\ldots,dt_n))=
\int_{u\in[0,1]} \eta((1-u)t_0, \ldots, (1-u)t_a+u, \ldots
(1-u)t_n, du, dt_a).\] We refer to the papers
\cite{getzler,navarro} for a proof of the identities $IE={\rm
Id}_{{\rm Tot}({\mathfrak g}^\Delta)}$ and $ EI-{\rm Id}_{{\rm
Tot}_{TW}({\mathfrak g}^\Delta)}=[h,d_{TW}]$. Here we point out
that $E$ and $h$ are defined in terms of integration over standard
simplices and multiplication with canonical differential forms: in
particular the construction of $\operatorname{Tot}_{TW}({\mathfrak
g}^\Delta)$, $\operatorname{Tot}({\mathfrak g}^\Delta)$, $I$, $E$
and $h$ is functorial in the category $[\mathbf{\Delta_{
mon}},\mathbf{DGLAs}]$ of semicosimplicial DGLAs.

\bigskip
 We are now in position to use the \emph{homotopical
transfer of structure} theorem to induce on $\Tot(\g^\Delta)$ an $L_\infty$-structure from the DGLA structure on $\Tot_{TW}(\g^\Delta)$. 
To state and discuss a few aspects of this transfer, it is convenient to use the decalage isomorphism and look at the $n$-ary brackets of an $L_\infty$ algebra as graded symmetric multilinear maps
\[
q_n\colon {\rm Sym}^n(V[1])\to V[2]
\]
With this convention, the homotopy transfer theorem reads as follows (see, e.g., \cite{fuka,HK,Kad,KoSo,KonSoi}).
\begin{theorem*}\label{thm.transfer}
Let $(V,q_1^{},q_2{},q_3^{},\dots)$ be an $L_\infty$-algebra and
$(C,\delta)$ be a differential complex. If there exist two
morphisms of differential complexes
\[
E\colon (C[1],\delta_{[1]}) \to (V[1],q_1) \qquad \text{and}
\qquad I\colon (V[1],q_1)\to (C[1],\delta_{[1]})
\]
such that the composition $EI$ is homotopic to the identity, then
there exist an $L_\infty$-algebra structure
$(C,\hat{q}_1,\hat{q}_2,\dots)$ on $C$ extending its differential
complex structure and  an $L_\infty$-morphism $E_\infty^{}$
extending $E$. In particular, if $E$ is a quasi-isomorphism of
complexes, then $E_\infty$ is a quasi-isomorphism of
$L_\infty$-algebras.
\end{theorem*}

It has been remarked by Kontsevich and Soibelman in \cite{KonSoi}
(see also \cite{fuka}) that the $L_\infty$-morphism $E_\infty$ and
the brackets $\hat{q}_n$ can be explicitly written as summations
over rooted trees. Let $h\in\Hom^{-1}(V[1],V[1])$ be an homotopy
between $EI$ and $\operatorname{Id}_{V[1]}$, i.e.,
$q_1h+hq_1=EI-\Id_{V[1]}$, and denote by ${\mathcal T}_{h,n}$ be
the groupoid whose objects are directed rooted trees with internal
vertices of valence at least two and exactly $n$ tail edges; trees
in ${\mathcal T}_{h,n}$ are  decorated as follows: each tail edge
of a tree in ${\mathcal T}_{h,n}$ is decorated by the operator
$E$, each internal edge is decorated by the operator $h$ and also
the root edge is decorated by the operator $h$; every internal
vertex $v$ carries the  operation $q_r$, where $r$ is the number
of edges having $v$ as endpoint. Isomorphisms between objects in
${\mathcal T}_{h,n}$ are isomorphisms of the underlying trees.
Denote the set of isomorphism classes of objects of ${\mathcal
T}_{h,n}$ by the symbol $T_{h,n}$. Similarly, let ${\mathcal
T}_{I,n}$ be the groupoid whose objects are directed rooted trees
with the same decoration as ${\mathcal T}_{h,n}$ except for the
root edge, which is decorated by the operator $I$ instead of $h$.
The set of isomorphism classes of objects of ${\mathcal T}_{I,n}$
is denoted $T_{I,n}$. Via the usual operadic rules, each decorated
tree $\Gamma\in {\mathcal T}_{h,n}$ gives a linear map
\[ Z_\Gamma(E,I,h,q_i)\colon C[1]^{\odot n}\to V[1],\]
similarly, each decorated tree $\Gamma \in {\mathcal T}_{I,n}$
gives rise to a degree one multilinear operator on $C[1]$ with
values in $C[1]$.\par

Having introduced these notations, we can write
Kontsevich-Soibelman's formulas as follows.
\begin{proposition*}\label{prop.sumovertrees}
In the  above set-up the brackets $\hat{q}_n$, and the
$L_\infty$ morphism
$E_\infty$ can be expressed as sums over
decorated rooted trees via the formulas
\[
E_n=\sum_{\Gamma\in
T_{h,n}}\frac{Z_\Gamma(E,I,h,q_i)}{|\operatorname{Aut}\Gamma|};
\qquad
\hat{q}_n=\sum_{\Gamma\in
T_{I,n}}\frac{Z_\Gamma(E,I,h,q_i)}{|\operatorname{Aut}\Gamma|};\qquad
n\geq 2.
\]
\end{proposition*}

\begin{corollary*} 
Let $\g^\Delta$ be a semicosimplicial Lie algebra. Then there is a natural isomorphism 
\[
\MC(\widetilde{\Tot}(\g^{{\Delta}}))\xrightarrow{\sim}Z^1(\exp(\g^\Delta))
\]
of functors $[\mathbf{Art}_\K,\mathbf{Sets}]$ over $\mathfrak{g}_1$.
\end{corollary*}
\begin{proof}
Recall from the proof of Theorem \ref{Th.Iso gruppoidi} that we have a natural isomorphism of functors
$
\Phi\circ E_\infty\colon \MC(\widetilde{\Tot}(\g^{{\Delta}}))\xrightarrow{\sim}Z^1(\exp(\g^\Delta))
$,
 so we just have to prove that the diagram
\[
\xymatrix{
\MC(\widetilde{\Tot}(\g^{{\Delta}}))\ar[rr]^{\Phi\circ E_\infty}\ar[rd]_{\rm incl.}&&Z^1(\exp(\g^\Delta))
\ar[ld]^{\rm incl.}\\
&\g_1&
}
\]
commutes, i.e.,
that the composition $\Phi \circ E_{\infty}$
is ``the identity'' on elements in
$\MC(\widetilde{\Tot}(\g^\Delta))$. Let $x$ be such an element; by definition,
$\Phi$ reads only the $(\Omega_1\otimes \g_1)$-component of
$E_{\infty}(x)$. We have $(E_1(x))_1=E(x)_1= (t_0 \, dt_1-t_1\, dt_0)x$,
which, under the isomorphism $\Omega_1\simeq \C[t,dt]$ reads
$(E_1(x))_1=-x \, dt$. If $n\geq2$, the formulas for $E_n$ involve the subgraph
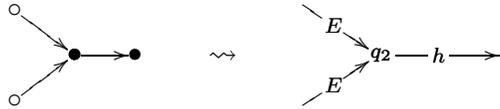
\[
\begin{xy}
,(-8,6)*{\circ};(0,0)*{\bullet}**\dir{-}?>*\dir{>}
,(-8,-6)*{\circ};(0,0)*{\bullet}**\dir{-}?>*\dir{>}
,(0,0)*{\bullet};(8,0)*{\bullet}**\dir{-}?>*\dir{>}
\end{xy}
\qquad
\rightsquigarrow
\qquad
\begin{xy}
,(-10,6.66);(-6,4)*{\,\scriptstyle{E}\,}**\dir{-}
,(-10,-6.66);(-6,-4)*{\,\scriptstyle{E}\,}**\dir{-}
,(-6,4)*{\,\scriptstyle{E}\,};
(0,0)*{\,\,\scriptstyle{q_2}\,}**\dir{-}?>*\dir{>}
,(-6,-4)*{\,\scriptstyle{E}\,};(0,0)*{\,\,\scriptstyle{q_2}\,}**\dir{-}?>*\dir{>}
,(0,0)*{\,\,\scriptstyle{q_2}\,};
(8,0)*{\,\scriptstyle{h}\,}**\dir{-}
,(8,0)*{\,\scriptstyle{h}\,};
(16,0)**\dir{-}?>*\dir{>}
\end{xy},
\]
where a white dot denotes an external vertex (a leaf) and a black dot an internal one, and so they involve the operation $q_2(E \otimes E)$ on ${\rm Sym}^2(\g_1)$.
Since we have:
\[
E: \g_1 \to \prod_{i\geq 1}\Omega^1_i\otimes \g_i,
\qquad\qquad q_2: {\rm Sym}^2(\prod_{i\geq 1}\Omega^1_i\otimes \g_i) \to \prod_{i\geq 2}\Omega^2_i\otimes \g_i,\]
and
\[
h\colon \prod_{i\geq 2}\Omega^2_i\otimes \g_i\to \prod_{i\geq 2}\Omega^1_i\otimes \g_i,
\]
the element
$E_n(x,x,\dots,x)$ has zero $(\Omega_1\otimes \g_1)$-component, for
$n\geq2$. Then $\Phi E_{\infty}(x)= \Phi E_1(x)= \Phi (-x dt)=
\Phi(e^{t x}*0)= x$.
\end{proof}

\end{document}